\documentclass[12pt,reqno]{amsart}
\usepackage{amssymb}
\usepackage{amsmath}
\usepackage{amsfonts}
\usepackage{mathrsfs}
\usepackage[dvipsnames]{xcolor}
\definecolor{verydarkblue}{HTML}{000099}
\usepackage{color}
\usepackage{enumerate}
\usepackage[T1]{fontenc}
\usepackage[latin1]{inputenc}
\usepackage{ae,aecompl}
\usepackage{graphicx}
\usepackage{nicefrac}
\usepackage[includehead,includefoot,margin=19mm]{geometry}
\usepackage{bigstrut}
\usepackage{hyperref}
\hypersetup{
    colorlinks=true,       
    linkcolor=verydarkblue,          
    citecolor=verydarkblue,        
    filecolor=verydarkblue,      
    urlcolor=verydarkblue           
}

\usepackage{amssymb}
\usepackage{amsmath}
\usepackage{amsfonts}
\usepackage{mathrsfs}
\usepackage[all]{xy}
\usepackage{hyperref}  
\usepackage{tikz}
\usepackage{verbatim}
\usepackage{mathrsfs}

\theoremstyle{plain}
\newtheorem{theorem}{Theorem}[section]

\newtheorem{lemma}[theorem]{Lemma}

\newtheorem*{theorem*}{Theorem}

\theoremstyle{definition}

\theoremstyle{remark}

\def\ZZ{{\mathbb Z}}

\def\O{{\mathcal O}}

\def\spec{{\mathrm{Spec}}}

\def\ker{{\mathrm{ker}}}

\def\G{{\mathbb{G}}}

\title
{Models of affine curves and $\mathbb{G}_{a}$-actions}
\author{Kevin Langlois}
\date{}

\address{Mathematisches Institut, Heinrich Heine Universit\"at, 40225 D\"usseldorf, Germany.}
\email{langlois.kevin18@gmail.com}

\begin{document}

\thanks{\em 2010 Mathematics Subject Classification \rm  14L15, 14L30. 
}

\maketitle
\begin{footnotesize}

\hspace{6.5cm}  {\em Dedicated to Mikhail Zaidenberg on  } 

\hspace{6.5cm} {\em  the occasion of his 70-th birthday}
\end{footnotesize}
\begin{abstract}
Using the approach of Barkatou and El Kaoui, we classify certain affine curves over discrete valuation rings having a free additive group action.
Our classification generalizes results of Miyanishi in equi-characteristic $0$.
\end{abstract}

\section{Introduction}
Let $\O$ be a discrete valuation ring. Choose a  uniformizer $t\in \O$ such that $k = \O/(t)$ is the residue field, and write $K$ for the  fraction field of $\O$. 
A faithful flat integral affine scheme of finite type over $\O$ is an \emph{affine $\O$-curve} if it has relative dimension $1$ and if $K$ is algebraically closed in its function field. Our aim is to classify the models of the affine line (i.e., affine $\O$-curves whose generic fiber is isomorphic to  $\mathbb{A}^{1}_{K}$). In particular, the arithmetic surface in question inherits a non-trivial action
of the additive group scheme $\G_{a, \O}$. 

 Miyanishi described
any affine $\O$-curve with a (free) $\G_{a, \O}$-action, such that the special fiber is integral, and under the condition that $\O$ is equi-characteristic $0$ \cite[Theorem 4.3]{Miy09}.
 Barkatou and El Kaoui extended this result in 
\cite{BE12} for reduced special fibers 
over an equi-characteristic $0$ principal ideal domain. 
Using the approach of \emph{loc. cit.}, we obtain the following generalization
which is valid in any characteristic.  
\begin{theorem}\label{tt-main}
Assume that $k$ is perfect.
Let $C$ be an affine $\O$-curve with a free $\mathbb{G}_{a, \O}$-action
and reduced special fiber. Then there exist a natural number $n\geq 1$ and polynomials $f_{i}\in \O[x_{1},\ldots, x_{i}]$
for $1\leq i\leq n$ such that 
$$C\simeq \spec\,\O[x_{1},\ldots, x_{n+1}]/\left(tx_{2}-f_{1},\ldots, tx_{n+1} - f_{n}\right).$$
Moreover, the following properties are fulfilled.
\begin{itemize}
\item[(i)] The ideal $\left(t, f_{1},\ldots,f_{n}\right)$ is $0$-dimensional and radical, and the reduction modulo $t$ of the $\partial f_{i}/ \partial x_{i}$'s are invertible.
\item[(ii)] Consider the subalgebra $B = \O[\alpha_{1}(x), \ldots, \alpha_{n}(x)]\subseteq K[x]$, where the $\alpha_{i}$'s are defined as $\alpha_{1}(x) = x$ and $\alpha_{i}(x) = t^{-1}f_{i-1}(\alpha_{1}(x), \ldots, \alpha_{i-1}(x))$ for $2\leq i\leq n$. Then, under the previous isomorphism, the $\G_{a, \O}$-actions on $C$ are in one-to-one correspondence with the $\G_{a, K}$-actions
$x\mapsto x  + \sum_{j = 1}^{r} c_{j} \lambda^{p^{s_{j}}}$
 on $\mathbb{A}_{K}^{1} = \spec\, K[x]$ that let stable the algebra $B$, where $s_{j}\in \ZZ_{\geq 0}, c_{j}\in K$, and $p$ is the characteristic exponent of the field $K$.
\end{itemize}  
\end{theorem}  
Section \ref{s-basics} sets the notation of the paper, while the proof of Theorem \ref{tt-main} is in Section \ref{s-last}.

\section{Basics}\label{s-basics} 
A $\mathbb{G}_{a, \O}$-action on the affine $\O$-curve $C= \spec\, B$
is equivalent to a sequence 
$$\delta^{(i)}:B\rightarrow B, \,\ i = 0,1,2, \ldots$$
of $\O$-linear maps sharing the conditions (see \cite{Miy68}):
\begin{itemize}
\item[(a)] The map $\delta^{(0)}$ is the identity,
\item[(b)] for any $b\in B$ there is $i\in\mathbb{Z}_{>0}$
such that $\delta^{(j)}(b) = 0$ for any $j\geq i$, 
\item[(c)] we have the \emph{Leibniz rule} 
$$\delta^{(i)}(b_{1}\cdot b_{2}) = \sum_{i_{1} + i_{2} = i}\delta^{(i_{1})}(b_{1})\cdot \delta^{(i_{2})}(b_{2}),$$
where $i\in \mathbb{Z}_{\geq 0}$ and $b_{1}, b_{2}\in B$, and
\item[(d)] for all indices $i, j\in \mathbb{Z}_{\geq 0}$: 
$$\delta^{(i)}\circ \delta^{(j)} = \binom{i+j}{i} \delta^{(i+j)}.$$
\end{itemize}
The sequence $\delta = (\delta^{(i)})$ is called a \emph{locally finite iterative
higher derivations} (LFIHD). The kernel $\ker(\delta)$ is the intersection of the linear subspaces  $\ker(\delta^{(i)})$ where $i$ runs over $\mathbb{Z}_{>0}$. Since $K$ is algebraically closed in the fraction field of $B$, we have $\ker(\delta) =  \O$ if the action is nontrivial. 
The $\G_{a, \O}$-action on $C$ is \emph{free} if the ideal generated by $\{\delta^{(i)}(b)\, ; i\in\mathbb{Z}_{>0}\text{ and }b\in B\}$
is $B$. The \emph{exponential morphism} is 
$$\exp(\delta T): B\rightarrow B[T],\,\, b\mapsto \sum_{i\in\mathbb{Z}_{\geq 0}}\delta^{(i)}(b) T^{i}.$$
\begin{lemma}\label{lem1} Assume that the $\mathbb{G}_{a, \O}$-action on $C$ is free.
Then there exist an LFIHD $\delta$ on $B$ corresponding to a free action and $x\in B$ such that 
$B\otimes_{\O}K = K[x]$ and $$\exp(\delta T)(x) = x + \sum_{i = 1}^{m}t^{n_{i}}T^{e_{i}}$$ for some natural numbers $n_{i}$ and some powers $$1 \leq e_{1}= p^{r_{1}}<\ldots <e_{m} = p^{r_{m}} =e,$$
where $p$ is the characteristic exponent of $\O$. Moreover, for any $\kappa\gg 0$ we may choose $\delta$  such that $\delta^{(e)}$ is an $\O$-derivation on the monomials of $K[x]$ of degree less than $\kappa$. Assume further that the special fiber of $C$ is reduced.
Then $B = \O[x]$ provided that $\min_{1\leq i\leq m}n_{i} =0$.
\end{lemma}
\begin{proof}
We may assume that $p >1$.
Let $\delta_{1}$ be the LFIHD defined by the $\G_{a,\O}$-action. 
Choose $x\in B$ such that $\exp(T\delta_{1})(x)$ has positive minimal degree. Then $B\otimes_{\O}K = K[x]$ \cite[Lemma 2.2 (c)]{CM05}. 
As the extension $\delta_{K}$ of $\delta_{1}$
on $K[x]$ corresponds to a $\G_{a, K}$-action on $\mathbb{A}^{1}_{K}$, we have  
$\delta^{(j)}_{1}(x)\in \ker(\delta_{K})\cap B = \O$
for any $j\in\mathbb{Z}_{>0}$. So if $\delta^{(j)}_{1}(x)\neq 0$, then $\delta^{(j)}_{1}(x) = c_{j}t^{m_{j}}$ for some $m_{j}\in \mathbb{Z}_{\geq 0}$
and $c_{j}\in \O^*$.  Consequently, we modify $\delta_{1}$ by changing the $c_{j}$'s by $1$.
Now write 
$$\exp(\delta_{1} T)(x) =  x + \sum_{i = 1}^{m-1}t^{n_{i}}T^{e_{i}}, \text{ where } m\geq 1.$$ We introduce a new LFIHD $\delta$ on $K[x]$ (trivial on $K$) defined by 
$$\exp(\delta T)(x) = \exp(\delta_{1} T)(x) + t^{n}T^{e},$$
where $e$ and $n$ satisfy the following conditions. Let $b_{1}, \ldots, b_{s}\in B$
such that $B = \O[b_{1},\ldots, b_{s}]$ and consider a relation
$$1 = \sum_{j = 1}^{\alpha}c_{j}\delta^{(\beta_{j})}_{1}(d_{j}) \text{ for } c_{j}, d_{j}\in  B \text{ and }\beta_{j}\in\ZZ_{>0},$$
which is guaranteed from the freeness assumption. Let $\kappa$ be a constant greater than the degrees in $x$ of the $b_{j}\in K[x]$ and take $e$ a power of $p$ verifying
$$e >(\kappa + 2)\max \left\{e_{j}, \beta_{\ell}\,\,| 1\leq j <m\text{ and } 1\leq \ell\leq \alpha\right\}.$$
Finally, let $n\in \ZZ_{>0}$ such that $t^{n}b_{j}\in \O[x]$ for $1\leq j\leq s$.
We claim that $\delta$ induces an LFIHD on $B$ with the required properties.
Indeed, if $i < e$, then $\delta^{(i)}(b) = \delta^{(i)}_{1}(b)$ for any $b\in B$. Now
assume that $i\geq e$ and let $$b_{j} = \sum_{u}^{d}\lambda_{u}x^{u},\text{ where }\lambda_{u}\in K \text{ and } d ={\rm deg}_{x}(b_{j}).$$ Let $e_{\delta}$ reaching the maximum of the $e_{j}$'s.
By a direct induction on $u\leq d$, $t^{n}$ divides $\delta^{(\beta)}(x^{u})$ if
$\beta> u e_{\delta}$ (note that $\delta^{(\beta)}(x^{u}) = ux^{u-1}t^{n}$ if $\beta = e$). Thus
$\delta^{(i)}(b_{j})\in \O[x]$ for any $j$ and $\delta$ induces a free $\G_{a, \O}$-action on $C$. This yields the first claim.

Let us show the second one.
The assumption $\min_{1\leq i\leq m}n_{i} =0$ implies that $\delta^{(\gamma)}(x) =1$ for some $\gamma$ and that the residue class $\bar{x}$ of $x$ modulo $t$ 
is not algebraic over $k$. Indeed, if $\bar{x}$ would admit an algebraic dependence relation, then, applying the exponential map (from the $\G_{a, k}$-action on $\spec\, B/tB$) to this relation, we would get a contradiction. 
Let $b\in B\setminus tB$. Since $B\subseteq K[x]$, there is a primitive polynomial $s(T)\in\O[T]$ such that 
$s(x) = t^{r}b$ for some $r\in\mathbb{Z}_{\geq 0}$. Observe that $s(x)\equiv 0 \,{\rm mod}\, tB$ if $r>0$. So the previous step implies that $r = 0$ and $b\in \O[x]$. Now let $c\in tB$. Write $ c= t^{\ell}a$ for some $\ell\in\mathbb{Z}_{>0}$ and $a\in B\setminus tB$. As $a\in\O[x]$,
we have $c\in \O[x]$. Thus $B = \O[x]$, as required.  
\end{proof}

\section{Proof of the main result}\label{s-last}
Let $C  =  \spec\, B$ be an affine $\O$-curve with a free $\mathbb{G}_{a, \O}$-action. Assume that the special fiber is reduced and that $k$ is perfect.
Let $\delta$ be an LFIHD on $B$ as in the proof of Lemma \ref{lem1}. Set $e:= e_{m}$ and $n  := n_{m}$.   
Consider $B_{1} := \O[x], I_{1} := (tB)\cap B_{1}$ and inductively define the other rings and ideals by
$$ B_{i+1} = B_{i}[t^{-1}I_{i}] \text{ and } I_{i+1} := (tB)\cap B_{i}\text{ for }i=1,\ldots,n.$$
The inclusions $B_{i}\subseteq B_{i+1}$ yield a sequence of $\G_{a, \O}$-equivariant \emph{affine modifications} (cf. \cite{KZ99})
$$\spec\, B_{n+1}\rightarrow \spec\, B_{n}\rightarrow \ldots  \rightarrow \spec\, B_{2}\rightarrow \mathbb{A}_{\O}^{1} =  \spec\, B_{1},$$
where $\spec\,B_{i+1}$ is the complement of the hypersurface $\mathbb{V}(t)$ in the blow-up of $\spec\, B_{i}$
with center $I_{i}$.  
The next lemma is analogous to \cite[Lemma 4.4]{BE12}. We use the adapted
result \cite[Lemma 4.3]{BE12} for the ring $\O$ where $k$ needs to be perfect.
Note that in the argument of the proof we will vary $\delta$ and $e$, while the number $n$ will be constant in the entire paper. 
\begin{lemma}\label{l-tec}
There exist $x_{1},\ldots, x_{n+1}\in B$ and $f_{i}\in \O[T_{1}, \ldots , T_{i}]$ for $1\leq i\leq  n$
such that the following hold.
\begin{itemize}
\item[(i)] $x_{1} = x$, $t$ divides $f_{i}(x_{1},\ldots x_{i})$, and if $$x_{i+1} = t^{-1} f_{i}(x_{1},\ldots x_{i}),$$
then $t^{n-i+1}$ divides $\delta^{(e)}(x_{i})$ for an appropriate choice of $\delta$.
\item[(ii)] $B_{i} = \mathcal{O}[x_{1},\ldots, x_{i}]$, $I_{i} = (t, f_{1}(x_{1}), \ldots, f_{i}(x_{1},\ldots, x_{i}))\subseteq B_{i},$
and the reduction modulo $t$ of $\partial f_{i}/ \partial T_{i}$ is invertible.
\item[(iii)] We have $C = \spec\, B_{n+1}$.
\end{itemize}
\end{lemma}
\begin{proof} 
We show the existence of the $f_{i}$'s and (ii) by induction on $i$. We treat the case $i = 1$.
 Let $\delta_{k}$ be the LFIHD corresponding to the $\G_{a, k}$-action on $\spec\, B/tB$ and set $R:= \ker(\delta_{k})$. By our assumption and \cite[Lemmata 2.1, 2.2]{CM05},  the scheme $\spec\, R$ is $0$-dimensional and reduced. Therefore $B_{1}/I_{1}\subseteq R$. By \cite[Lemma 4.3]{BE12}, there is $f_{1}\in \O[T_{1}]$ such that $I_{1} = (t, f_{1}(x))\subseteq B_{1}$ and the reduction modulo $t$ of $\partial f_{1}/ \partial T_{1}$ is a unit. If $x_{2}:= t^{-1}f_{1}(x)$, then $B_{2} =  \O[x_{1}, x_{2}]$. 
Assume that Statement (ii) holds for $i< n$.  It follows that $B_{i+1} =\O[x_{1}, \ldots, x_{i+1}]$. Let $J$ be the preimage of $I_{i+1}$ by the morphism $E:= \O[T_{1},\ldots, T_{i+1}]\rightarrow B_{i+1},\, \, T_{j}\mapsto x_{j}.$  Since $\spec\, E/J$ is reduced and $0$-dimensional, by \cite[Lemma 4.3]{BE12} there exist polynomials $g_{j}\in \O[T_{1}\ldots, T_{j}]$ such that 
$$I_{i+1} = (t, g_{1}(x_{1}), \ldots, g_{i+1}(x_{1},\ldots, x_{i+1}))\subseteq B_{i+1}$$
and the reductions modulo $t$ of the $\partial g_{j}/ \partial T_{j}$'s are invertible.
From Property (ii) in \emph{loc. cit.} we may choose $g_{1},\ldots, g_{i}$ such that $I_{i} = (t, g_{1}(x_{1}), \ldots, g_{i}(x_{1},\ldots, x_{i}))$. Therefore, we take
$g_{j} = f_{j}$ for any $1\leq j\leq i$ (from our induction hypothesis) and let $f_{i+1} = g_{i+1}$, as required.

Choose an LFIHD $\delta$ with $e\gg 0$ as in Lemma \ref{lem1} such that 
$\delta^{(e)}$ acts as an $\O$-derivation on $x_{1}, \ldots, x_{n+1}$  (seen
as polynomials in $x$). We show (i) by induction on $i$. 
For $i = 1$, we have $\delta^{(e)}(x_{1}) = \delta^{(e)}(x) = t^{n}$, and for $i =2$, we get $\delta^{(e)}(f_{1}(x)) = t^{n}\partial f_{1}(x)/\partial x$. As $x_{2}:= t^{-1}f_{1}(x)$, it follows that $t^{n-1}$ divides $\delta^{(e)}(x_{2})$. 
Assume that Statement (i) holds for $i< n$. By induction hypothesis,
$t^{n-j+1}$ divides $\delta^{(e)}(x_{j})$ for any $j\leq i$. This implies 
$$\delta^{(e)}(x_{i+1}) = t^{-1} \left(\sum_{j = 1}^{i}\partial f_{j}/\partial x_{j} \cdot \delta^{(e)}(x_{j}) \right)$$
and so $t^{n-i}$ divides $\delta^{(e)}(x_{i+1}),$ proving (i).

(iii) If $n= 0$, then $B  =  \O[x] = B_{1}$ (see Lemma \ref{lem1}). Thus, we assume $n>0$.
Let $b\in B$ such that $tb\in B_{n+1}$ and write $tb = \sum_{j = 1}^{r}a_{j}x_{n+1}^{j}$ for some $a_{j}\in\O[x_{1},\ldots, x_{n}]$. By the reasonning we did before, we may choose $\delta$ as in Lemma \ref{lem1} such that $\delta^{(e)}$ acts as an $\O$-derivation
on the polynomials $a_{j}x_{n+1}^{j}.$
 We show by induction on $r$ that $a_{j}\in I_{n}$ for any $j$. The case $r =0$ being obvious, assume that the statement 
holds true for $r-1$. Then
$$t\delta^{(e)}(b) = \sum_{j = 1}^{r}\delta^{(e)}(a_{j})x_{n+1}^{j} + \delta^{(e)}(x_{n+1})\cdot \left(\sum_{j = 1}^{r}ja_{j}x_{n+1}^{j-1}\right).$$
By Statement (i) and the fact that $n>0$, $t$ divides the $\delta^{(e)}(a_{j})$'s. From a direct computation,
$$\delta^{(e)}(x_{n+1}) = t\alpha  + \prod_{i =1}^{n}\partial f_{i} / \partial x_{i} \text{ for some }\alpha\in B.$$
Therefore $t$ divides $\sum_{j = 1}^{r}ja_{j}x_{n+1}^{j-1}$. Let $p$ be the characteristic exponent of $\O$ and assume that $p>1$. 
By induction assumption, $a_{j}\in I_{n}$ for any nonzero $j\not\in p\mathbb{Z}$.
Now, there is $b'\in B$ such that 
$$tb' = \sum_{1\leq j \leq r, j\not\in p\mathbb{Z}}a_{j}x_{n+1}^{j}.$$ Letting $b_{1} = b-b'$ we may write 
$$tb_{1} = \sum_{u= 1}^{s}a_{up^{\ell}}(x_{n+1}^{p^{\ell}})^{u} \text{ for some } s, \ell \geq 1.$$ 
Here $\ell$ is taken so that $a_{u p^{\ell}}\neq 0$ for some $1\leq u\leq s$ with $u\not\in p \mathbb{Z}$.
Now applying $\delta^{(ep^{\ell})}$ we obtain 
$$t\delta^{(ep^{\ell})}(b_{1}) =  \sum_{u =1}^{s}\delta^{(ep^{\ell})}(a_{up^{\ell}})(x_{n+1}^{p^{\ell}})^{u} + (\delta^{(e)}(x_{n+1}))^{p^{\ell}}\cdot \left( \sum_{u =1}^{s}ua_{up^{\ell}}(x_{n+1}^{p^{\ell}})^{u-1}\right).$$
By the Leibniz rule and the fact that $n>0$, $t$ divides the $\delta^{(ep^{\ell}))}(a_{up^{\ell}})$'s. Thus  
$$t \text{ divides }\sum_{u =1}^{s}ua_{up^{\ell}}(x_{n+1}^{p^{\ell}})^{u-1}.$$
According to our induction hypothesis, $a_{up^{\ell}}\in I_{n}$ for any $u\not\in p\mathbb{Z}$. Continuing this process, we arrive at
$a_{j}\in I_{n}$ for any $j$. Let us show that $B = B_{n+1}$. By the previous step, we have
$$a_{j} = a_{0, j}t+ \sum_{i= 1}^{n}a_{i,j}f_{i} \text{ for } a_{i,j}\in B_{n}\text{ and } j= 1,\ldots r.$$
Setting $c_{i} = \sum_{j =1}^{r}a_{i,j}x_{n+1}^{j},$  we have 
$$b = t^{-1}\left(\sum_{j = 1}^{r}a_{j}x_{n+1}^{j}\right) = \sum_{i = 1}^{n}c_{i}x_{i}\in B_{n+1}, \text{ where }x_{0} =1.$$
Hence $b\in B_{n+1}$ provided that $t^{\varepsilon}b\in B_{n+1}$ for some $\varepsilon\in\mathbb{Z}_{\geq 0}.$ From
the equality $B_{(t)} = K[x] = (B_{n+1})_{(t)}$ we conclude that $B = B_{n+1}$. This completes the proof of the lemma.
\end{proof}
\begin{proof}[Proof of Theorem \ref{tt-main}.] We follow the proof of \cite[Theorem 3.1]{BE12}. 
Consider the surjective morphism
$\psi: E:= \O[T_{1},\ldots, T_{n+1}]\rightarrow B_{n+1} = B, \, T_{i}\mapsto x_{i},$
the ideal $I = (tT_{2} - f_{1}, \ldots, tT_{n+1} - f_{n})$ and let $J  =  \bigcup_{i} (I: t^{i}E).$
We show that $J = \ker\, \psi$. Note that $J\subseteq \ker\, \psi$ is clear. Let $b\in \ker\, \psi$. By performing several 
Euclidean divisions we get an equality
$$t^{\ell}b = \sum_{j =1}^{r}\gamma_{j}(tT_{j+1} - f_{j}) + \beta, \text{ where }\beta\in \O[T_{1}], \ell\in\mathbb{Z}_{\geq 0}, \text{ and } \gamma_{j}\in E.$$
Since $\psi(t^{\ell}b) = 0$ and $x_{1} = x$ is transcendental over $K$, we have $\beta =0$. Thus $\ker\, \psi = J$. 
It remains to prove that $I = J$. From \cite[Lemma 4.5]{BE12} (which is valid in our context) we only need to have $J\subseteq tE + I = (t, f_{1},\ldots, f_{n})$.
Let $b\in J$. Then $b = \sum_{j = 0}a_{j}T_{n+1}^{j}$ for some $a_{j}\in \O[T_{1},\ldots, T_{n}]$. Since $\psi(b) = 0$, the argument of the proof
of Lemma \ref{l-tec} (iii) implies that $\psi(a_{j})\in I_{n}$ for any $j$. Thus $b\in tE + I$, establishing the theorem.  
\end{proof}   
\ {\em Acknowledgments.} 
The author was supported by the Heinrich Heine University of D\"usseldorf. This research is supported by ERCEA Consolidator Grant 615655 - NMST and also by the Basque Government through the BERC 2014-2017 program and by Spanish Ministry of Economy and Competitiveness MINECO: BCAM Severo Ochoa excellence accreditation SEV-2013-0323.


\begin{thebibliography}{}
\bibitem[BE12]{BE12}
Moulay A. Barkatou and M'hammed El Kahoui.  Locally nilpotent derivations with a PID ring of constants. Proc. Amer. Math. Soc. 140 (2012), no. 1, 119--128.
\bibitem[CM05]{CM05} Anthony J. Crachiola and Leonid G. Makar-Limanov. On the rigidity of small domains. J. Algebra 284 (2005), no.1 1--12.
\bibitem[KZ99]{KZ99} Shulim Kaliman and Mikhail Zaidenberg. Affine modifications and affine hypersurfaces with a very transitive automorphism group, Transform. Groups 4 (1999), no.1, 53--95.
\bibitem[Miy68]{Miy68}
Masayoshi Miyanishi.  A remark on an iterative infinite higher derivation. J. Math. Kyoto Univ. 8 1968 411--415.

\bibitem[Miy09]{Miy09}
Masayoshi Miyanishi.  Additive group scheme actions on integral schemes defined over discrete valuation rings. J. Algebra 322 (2009), no. 9, 3331--3344.

\end{thebibliography}
\end{document}